\documentclass{amiart}

\newtheorem{proposition}[theorem]{Proposition}
\definecolor{CORbleu}{RGB}{0,0,0}
\newcommand{\BLU}[1]{\textcolor{CORbleu}{#1}}

\title{Delannoy--Steinhaus triangles over $\mathbb{Z}/2\mathbb{Z}$: weight spectrum, balanced triangles, and extremal values}
\author{ Hac\`ene Belbachir\inst1, Randa Ouchene\inst1}
\institute{
  \inst1 \BLU{USTHB, Faculty of Mathematics, RECITS Laboratory, Algiers, Algeria}\\
  \url{randaaouchene@gmail.com},\url{randa.ouchene1@uqac.ca}, \url{hbelbachir@usthb.dz}\BLU{, \url{hacenebelbachir@gmail.com}}
}

\begin{document}
\maketitle

\begin{abstract}
A Delannoy--Steinhaus triangle is obtained from a finite sequence by a
recurrence governed by the Delannoy numbers. We introduce this construction
over $\mathbb{Z}/2\mathbb{Z}$ and study its weight distribution. The relevant
Delannoy coefficients are all odd, which reduces every entry to the parity of
a consecutive interval of the generating sequence. Encoding these interval
parities by prefix parities yields a weight formula depending only on the
numbers of zeros and ones in the prefix-parity sequence. We use this formula
to determine the complete weight spectrum and the exact multiplicity of each
weight. As a consequence, we characterize and enumerate the balanced
triangles: a balanced triangle generated by a binary sequence of length $n$
exists if and only if $n+1$ is a perfect square. We also determine the
canonical-vector weights, the minimum nonzero weight, the
{second-smallest nonzero weight}, the maximum weight, and the
average weight.


\end{abstract}

\section{Introduction}

Delannoy numbers count lattice paths using horizontal, vertical, and diagonal
unit steps. Their combinatorial origin and usefulness in lattice-path
enumeration are surveyed by Banderier and Schwer~\cite{BanderierSchwer2005}.
Generalized Delannoy numbers and their recurrences were investigated by
Edwards and Griffiths~\cite{EdwardsGriffiths2020}, while matrix arrangements
of these numbers exhibit further recurrence and unimodality properties
\cite{AmroucheBelbachirRamirez2019}.

Steinhaus triangles arise from {the local rule of Pascal's
triangle}. Starting with a finite sequence, each successive row is obtained by
adding consecutive terms.
Steinhaus introduced the binary construction and asked for balanced triangles,
that is, triangles containing equally many zeros and ones
\cite{Steinhaus1963}. Harborth gave the first complete constructive solution
of this problem~\cite{Harborth1972}. {Related generalizations in
which each new entry is the sum of a fixed number of consecutive entries were
studied by Bode and Harborth~\cite{BodeHarborth2014}.} Later work imposed additional structure
on the generating sequence, including symmetry
\cite{BrunatMaureso2011}, zero-sum conditions
\cite{EliahouMarinRevuelta2007}, and periodicity
\cite{Chappelon2017}.

The balance question was also extended from the binary field to finite cyclic
groups. Molluzzo formulated the corresponding modular problem through
Steinhaus graphs and triangles~\cite{Molluzzo1978}. Chappelon established
infinite families in finite cyclic groups and constructed a universal
sequence producing balanced Steinhaus figures modulo odd integers
\cite{Chappelon2008,Chappelon2011}. More recently, infinitely many balanced
Steinhaus triangles were obtained for every modulus
\cite{Chappelon2025}. These results concern the multiplicities of all residues,
whereas in the binary setting balance is equivalent to prescribing one
particular weight.

A second line of research studies the weight distribution itself. Brunat and
Maureso determined several smallest and largest attainable weights
\cite{BrunatMaureso2016,BrunatMaureso2017Weights} and developed a method for
computing the weights of triangles generated by canonical basis vectors
\cite{BrunatMaureso2018}. Symmetric binary triangles and their relationship
with parity-regular Steinhaus graphs were subsequently described by
Chappelon~\cite{Chappelon2022}. To the best of our knowledge,
the available results for the ordinary Pascal rule provide selected extremal
values or structured families, but not a single formula describing the
complete weight spectrum.

The present paper combines the Delannoy and Steinhaus directions.
{More precisely, the binomial propagation coefficients occurring
in an ordinary Steinhaus triangle are replaced by the coefficients in a
Delannoy row.} This is not a reformulation of the classical Pascal rule.
Nevertheless, over
$\mathbb{Z}/2\mathbb{Z}$ every relevant Delannoy coefficient is odd, and this
arithmetic fact makes each entry the parity of a consecutive interval of the
initial sequence. We encode all interval parities through prefix parities and
obtain an exact weight formula. It yields the complete spectrum, the number of
generators in each prefix-composition class, the canonical-vector weights,
the two smallest positive weights, the maximum, and the mean. It also gives
a complete solution to the balance problem for this construction: a balanced
triangle generated by a sequence of length $n$ exists precisely when $n+1$
is a perfect square, and in that case all such generators are characterized
and counted explicitly. Thus the
Delannoy recurrence leads to a complete distributional description that
complements the known extremal and balance results for ordinary Steinhaus
triangles.

The paper is organized as follows. Section~2 introduces Delannoy--Steinhaus
triangles over an arbitrary modulus, proves that the Delannoy coefficients are
odd (Lemma~\ref{lem:odd}), and expresses each entry as a Delannoy-weighted
combination of the generating sequence (Proposition~\ref{prop:entry}).
Section~3 specializes to $\mathbb{Z}/2\mathbb{Z}$, introduces the prefix-parity
encoding, and establishes the fundamental weight identity $|D(X)|=z(n+1-z)$
(Proposition~\ref{prop:prefix}). Section~4 uses this identity to determine the
complete weight spectrum together with the multiplicity of each weight
(Theorem~\ref{thm:spectrum}), the existence, characterization, and number of
balanced triangles (Corollary~\ref{cor:balanced}), the weights of the canonical generators
(Corollary~\ref{cor:canonical}), and the {smallest and
second-smallest nonzero weights, together with the largest weight}
(Theorem~\ref{thm:extreme} and Corollary~\ref{cor:minclass}). Section~5 computes
the average weight over all $2^n$ binary generators (Theorem~\ref{thm:mean}),
and Section~6 gathers concluding remarks and directions for further work.
\section{Delannoy--Steinhaus triangles}

Let $D_{u,v}$ be the Delannoy number counting lattice paths from $(0,0)$ to $(u,v)$ with steps $(1,0)$, $(0,1)$, and $(1,1)$. Define the triangular Delannoy array
\[
d_{i,k}=D_{i-k,k},
\qquad 0\leq k\leq i.
\]
It admits the explicit expression
\begin{equation}\label{eq:delannoy}
{
d_{i,k}=\sum_{\ell=0}^{\min(k,i-k)}
\binom{k}{\ell}\binom{i-k}{\ell}2^\ell.}
\end{equation}

\begin{lemma}\label{lem:odd}
For every $i\geq0$ and $0\leq k\leq i$, the coefficient $d_{i,k}$ is odd.
\end{lemma}

\begin{proof}
In the sum in~\eqref{eq:delannoy}, the term corresponding to $\ell=0$ is
\[
\binom{k}{0}\binom{i-k}{0}2^0=1.
\]
Every term with $\ell\geq1$ is divisible by $2$. Therefore
$d_{i,k}\equiv1\pmod2$, which proves that $d_{i,k}$ is odd.
\end{proof}

Let $m\geq2$ and let
\[
X=(a_{0,0},a_{0,1},\ldots,a_{0,n-1})
\in(\mathbb{Z}/m\mathbb{Z})^n.
\]
The Delannoy--Steinhaus triangle $D(X)=(a_{i,j})$ is defined for
$0\leq i\leq n-1$ and $0\leq j\leq n-i-1$ by
\begin{align*}
{a_{1,j}}&{=a_{0,j}+a_{0,j+1},
\qquad 0\leq j\leq n-2,}\\
{a_{i+1,j}}&{=a_{i,j}+a_{i,j+1}+a_{i-1,j+1},
\quad 1\leq i\leq n-2,\quad 0\leq j\leq n-i-2,}
\end{align*}
with all operations performed modulo $m$.

\begin{example}[A binary Delannoy--Steinhaus triangle]\label{ex:first-delannoy}
Work over $\mathbb{Z}/2\mathbb{Z}$ and take
\[
X=(1,0,1,1,0).
\]
The first derived row is obtained by adding adjacent entries. From the
second derived row onward, the three terms in the Delannoy recurrence
are used. For example,
\[
a_{2,0}=a_{1,0}+a_{1,1}+a_{0,1}
=1+1+0=0\pmod2.
\]
All the rows are
\[
\begin{array}{c@{\qquad}ccccc}
i=0&1&0&1&1&0\\
i=1&&1&1&0&1\\
i=2&&&0&0&0\\
i=3&&&&1&0\\
i=4&&&&&1
\end{array}
\]
and the triangle has weight
\[
|D(X)|=3+3+0+1+1=8.
\]
This example will also illustrate below how prefix parities count all
the ones without recomputing the rows.
\end{example}

\begin{proposition}\label{prop:entry}
For every valid pair $(i,j)$,
\begin{equation}\label{eq:entry}
a_{i,j}=\sum_{k=0}^{i}d_{i,k}a_{0,j+k}\pmod m.
\end{equation}
\end{proposition}

\begin{proof}
We use induction on the row index $i$. For $i=0$, formula~\eqref{eq:entry}
reduces to $a_{0,j}=d_{0,0}a_{0,j}$, and $d_{0,0}=D_{0,0}=1$. For
$i=1$, the defining rule gives
\[
a_{1,j}=a_{0,j}+a_{0,j+1}
=d_{1,0}a_{0,j}+d_{1,1}a_{0,j+1},
\]
because $d_{1,0}=d_{1,1}=1$.

Assume now that~\eqref{eq:entry} holds for rows $i$ and $i-1$, where
$i\geq1$. We adopt the convention that $d_{r,s}=0$ whenever
$s<0$ or $s>r$. Applying the recurrence that defines the triangle and then
the induction hypothesis, we obtain
\begin{align*}
a_{i+1,j}
&=a_{i,j}+a_{i,j+1}+a_{i-1,j+1}\\
&=\sum_{k=0}^{i}d_{i,k}a_{0,j+k}
  +\sum_{k=0}^{i}d_{i,k}a_{0,j+1+k}
  +\sum_{k=0}^{i-1}d_{i-1,k}a_{0,j+1+k}\\
&=\sum_{k=0}^{i+1}
\bigl(d_{i,k}+d_{i,k-1}+d_{i-1,k-1}\bigr)a_{0,j+k}.
\end{align*}
The Delannoy recurrence
\[
D_{u,v}=D_{u-1,v}+D_{u,v-1}+D_{u-1,v-1}
\]
implies, after substituting $d_{i,k}=D_{i-k,k}$, that
\[
d_{i+1,k}=d_{i,k}+d_{i,k-1}+d_{i-1,k-1}.
\]
Thus the last sum equals
$\sum_{k=0}^{i+1}d_{i+1,k}a_{0,j+k}$, which is precisely
formula~\eqref{eq:entry} for row $i+1$. This completes the induction.
\end{proof}

\section{Binary representation by prefix parities}

We now work over $\mathbb{Z}/2\mathbb{Z}$ and write the first row as
$X=(a_0,\ldots,a_{n-1})$. Since each $d_{i,k}$ is odd, equation~\eqref{eq:entry} becomes
\begin{equation}\label{eq:binary-entry}
a_{i,j}=\sum_{k=0}^{i}a_{j+k}\pmod2.
\end{equation}
Thus $a_{i,j}$ is the parity of the interval
$a_j,\ldots,a_{j+i}$.

\begin{definition}
The prefix-parity sequence associated with $X$ is
\[
P=(P_0,P_1,\ldots,P_n),
\qquad
P_0=0,\qquad
P_r=\sum_{h=0}^{r-1}a_h\pmod2
\quad(1\leq r\leq n).
\]
\end{definition}

The map $X\mapsto P$ is a bijection between binary sequences of length $n$
and binary sequences of length $n+1$ whose first term is $0$. Indeed, the
definition of $P$ gives
\[
a_r=P_r+P_{r+1}\pmod2.
\]
Hence $P$ determines every coordinate of $X$. Conversely, starting from any
binary sequence $P=(0,P_1,\ldots,P_n)$ and defining $a_r=P_r+P_{r+1}$,
the resulting prefix parities telescope to the prescribed values of $P$.

\begin{proposition}[Prefix-parity formula]\label{prop:prefix}
For every binary sequence $X$ of length $n$,
\[
|D(X)|
=\#\{(u,v):0\leq u<v\leq n,\ P_u\neq P_v\}.
\]
Consequently, if $P$ contains $z$ zeros and $n+1-z$ ones, then
\begin{equation}\label{eq:weight-z}
|D(X)|=z(n+1-z).
\end{equation}
\end{proposition}

\begin{proof}
By equation~\eqref{eq:binary-entry},
\[
a_{i,j}
=\sum_{k=0}^{i}a_{j+k}
=\sum_{r=j}^{j+i}a_r
=P_j+P_{j+i+1}\pmod2,
\]
where the last identity follows because all terms preceding position $j$
occur twice in the sum $P_j+P_{j+i+1}$ and therefore cancel modulo $2$.
Hence $a_{i,j}=1$ exactly when $P_j\neq P_{j+i+1}$. The valid entries $(i,j)$ are in bijection with all pairs $(u,v)$ satisfying $0\leq u<v\leq n$, through $u=j$ and $v=j+i+1$. If $P$ contains $z$ zeros and $n+1-z$ ones, each discordant pair consists of one zero and one one, giving $z(n+1-z)$ pairs.

For completeness, the correspondence between entries and pairs is bijective:
given $(u,v)$ with $u<v$, the inverse map is
\[
j=u,\qquad i=v-u-1.
\]
It satisfies $i\geq0$ and $j+i=v-1\leq n-1$, so $(i,j)$ is a valid
position of the triangle. Finally, to count pairs with unequal prefix
parities, choose one of the $z$ zero positions and one of the
$n+1-z$ one positions. These two positions have a unique increasing order,
so they determine exactly one pair $(u,v)$ with $u<v$. The number of such
pairs is therefore $z(n+1-z)$.
\end{proof}

\begin{remark}
For Example~\ref{ex:first-delannoy}, the prefix-parity sequence is
\[
P=(0,1,1,0,1,1).
\]
It contains two zeros and four ones. Formula~\eqref{eq:weight-z} therefore
gives $|D(X)|=2\cdot4=8$, in agreement with the displayed triangle.
\end{remark}

\section{Weight spectrum and extremal values}

Let
\[
\mathcal{W}_n=\{|D(X)|:X\in(\mathbb{Z}/2\mathbb{Z})^n\}.
\]

\begin{theorem}[Complete spectrum]\label{thm:spectrum}
For every $n\geq1$,
\[
\mathcal{W}_n=\{z(n+1-z):1\leq z\leq n+1\}.
\]
Moreover, the number of generating sequences whose prefix-parity sequence contains exactly $z$ zeros is
\[
\binom{n}{z-1}.
\]
More precisely, for $w\in\mathcal{W}_n$, let
\[
\mathcal{Z}_n(w)
=\{z\in\{1,\ldots,n+1\}:z(n+1-z)=w\}.
\]
Then the number of sequences whose triangle has weight exactly $w$ is
\begin{equation}\label{eq:weight-multiplicity}
\#\{X:|D(X)|=w\}
=\sum_{z\in\mathcal{Z}_n(w)}\binom{n}{z-1}.
\end{equation}
\end{theorem}

\begin{proof}
Let $X$ be a binary sequence and let $z$ be the number of zeros in its
prefix-parity sequence. Since $P_0=0$, we necessarily have
$1\leq z\leq n+1$. Proposition~\ref{prop:prefix} then shows that every
attainable weight belongs to the displayed set.

Conversely, fix any $z\in\{1,\ldots,n+1\}$. Keep $P_0=0$, choose
$z-1$ of the remaining $n$ positions $P_1,\ldots,P_n$ to contain zero,
and put one in every other position. This constructs a prefix-parity
sequence with exactly $z$ zeros. The bijection between $X$ and $P$ produces
a unique generator $X$, and Proposition~\ref{prop:prefix} gives its weight
as $z(n+1-z)$. Thus every value in the displayed set is attained.

There are exactly $\binom{n}{z-1}$ possible choices of the additional zero
positions, and the bijection shows that no two of them produce the same
generator. This proves the composition-class count. Finally, a generator has
weight $w$ precisely when its number of zero prefix parities belongs to
$\mathcal{Z}_n(w)$. The classes corresponding to different values of $z$ are
disjoint, so summing their cardinalities gives~\eqref{eq:weight-multiplicity}.
\end{proof}

\subsection{Balanced triangles}

A binary triangle is called \emph{balanced} when it contains equally many
zeros and ones.

\begin{corollary}[Balanced triangles]\label{cor:balanced}
Let $n\geq1$. A balanced Delannoy--Steinhaus triangle generated by a binary
sequence of length $n$ exists if and only if $n+1$ is a perfect square.
More precisely, {if $n+1=q^2$, then $D(X)$ is balanced if and only if the
prefix-parity sequence of $X$ contains}
\[
{z=\frac{q^2-q}{2}
\qquad\text{or}\qquad
z=\frac{q^2+q}{2}}
\]
{zeros}. The number of binary sequences of length $n$ generating a balanced
triangle is therefore
\begin{equation}\label{eq:balanced-count}
\#\{X:D(X)\text{ is balanced}\}
=
\begin{cases}
\displaystyle
\binom{n+1}{\frac{n+1-\sqrt{n+1}}{2}},
&\text{if $n+1$ is a perfect square},\\[3mm]
0,
&\text{otherwise}.
\end{cases}
\end{equation}
Equivalently, {when $n+1=q^2$, this number is}
\[
{\binom{q^2}{(q^2-q)/2}.}
\]
\end{corollary}

\begin{proof}
The triangle $D(X)$ has $n(n+1)/2$ entries, so it is balanced exactly when
\[
|D(X)|=\frac{n(n+1)}{4}.
\]
Put $M=n+1$. By Proposition~\ref{prop:prefix}, this condition is equivalent
to
\[
z(M-z)=\frac{M(M-1)}{4},
\]
or, after rearranging,
\begin{equation}\label{eq:balanced-square}
(2z-M)^2=M.
\end{equation}
Thus an integral value of $z$ exists if and only if {$M=q^2$ for some integer
$q\geq2$}. In that case equation~\eqref{eq:balanced-square} gives
\[
{z=\frac{q^2-q}{2}
\qquad\text{or}\qquad
z=\frac{q^2+q}{2}.}
\]
Both values are integers, and Theorem~\ref{thm:spectrum} shows that both are
attained.

Let {$z_-=(q^2-q)/2$ and $z_+=(q^2+q)/2$}. Since $z_+=M-z_-$, the number of
corresponding generators is, again by Theorem~\ref{thm:spectrum},
\begin{align*}
\binom{M-1}{z_--1}+\binom{M-1}{z_+-1}
&=\binom{M-1}{z_--1}+\binom{M-1}{z_-}\\
&=\binom{M}{z_-},
\end{align*}
which is the stated formula.
\end{proof}

\subsection{Canonical generators}

Let $e_k^{(n)}$ be the sequence with a single $1$ in position $k$.

\begin{corollary}\label{cor:canonical}
For $0\leq k\leq n-1$,
\[
|D(e_k^{(n)})|=(k+1)(n-k).
\]
\end{corollary}

\begin{proof}
For $X=e_k^{(n)}$, no $1$ has been encountered before position $k+1$.
Thus
\[
P_r=0\quad(0\leq r\leq k),
\qquad
P_r=1\quad(k+1\leq r\leq n).
\]
The prefix-parity sequence therefore contains $k+1$ zeros and $n-k$ ones.
Substituting $z=k+1$ into~\eqref{eq:weight-z} yields
$|D(e_k^{(n)})|=(k+1)(n-k)$.
\end{proof}

\subsection{Smallest and largest weights}

\begin{theorem}\label{thm:extreme}
For $n\geq3$, the smallest nonzero weight is
\[
w_1=n,
\]
and the {second-smallest nonzero weight} is
\[
w_2=2(n-1).
\]
The maximum weight is
\[
\max_X|D(X)|=\left\lfloor\frac{(n+1)^2}{4}\right\rfloor.
\]
\end{theorem}

\begin{proof}
Put $f(z)=z(n+1-z)$. By equation~\eqref{eq:weight-z}, the nonzero
weights are exactly the values $f(z)$ with $1\leq z\leq n$; the remaining
case $z=n+1$ gives the zero triangle. The function is symmetric because
\[
f(n+1-z)=f(z).
\]
Moreover,
\[
f(z+1)-f(z)=n-2z.
\]
Consequently, $f$ is strictly increasing as $z$ moves from $1$ toward the
middle of the interval. In particular, for $n\geq3$, the
{two smallest positive values} are
\[
f(1)=f(n)=n
\quad\text{and}\quad
f(2)=f(n-1)=2(n-1).
\]

To maximize $f(z)$, write
\[
f(z)=\frac{(n+1)^2}{4}
-\left(z-\frac{n+1}{2}\right)^2.
\]
The maximum over integral $z$ is therefore attained at the integer or the two
integers closest to $(n+1)/2$. If $n+1$ is even, the maximum is
$(n+1)^2/4$; if $n+1$ is odd, it is $((n+1)^2-1)/4$. In both cases this is
$\lfloor (n+1)^2/4\rfloor$.
\end{proof}

\begin{corollary}\label{cor:minclass}
The weight $n$ is attained exactly by
\[
\{e_0^{(n)},e_{n-1}^{(n)}\}
\cup
\{(0^k,1,1,0^{n-k-2}):0\leq k\leq n-2\}.
\]
The maximum is attained, in particular, by
$e_{\lfloor(n-1)/2\rfloor}^{(n)}$.
\end{corollary}

\begin{proof}
By the proof of Theorem~\ref{thm:extreme}, weight $n$ occurs exactly when
$z=1$ or $z=n$.

If $z=1$, then $P_0$ is the unique zero and
\[
P=(0,1,1,\ldots,1).
\]
Using $a_r=P_r+P_{r+1}$ gives $X=e_0^{(n)}$.

If $z=n$, the prefix-parity sequence contains a unique one, necessarily at
some position $j\in\{1,\ldots,n\}$. When $j=n$, taking successive differences
gives $X=e_{n-1}^{(n)}$. When $1\leq j\leq n-1$, the only changes in $P$
occur between positions $j-1$ and $j$, and between positions $j$ and $j+1$.
Hence
\[
a_{j-1}=a_j=1
\quad\text{and}\quad
a_r=0\ \text{for }r\notin\{j-1,j\}.
\]
Writing $k=j-1$ gives the family
$(0^k,1,1,0^{n-k-2})$, with $0\leq k\leq n-2$. These cases are exhaustive
and mutually distinct.

Finally, let $k=\lfloor(n-1)/2\rfloor$. Then the two factors in the canonical
weight $(k+1)(n-k)$ differ by at most one. Corollary~\ref{cor:canonical} and
the maximum calculation in Theorem~\ref{thm:extreme} therefore give
\[
|D(e_k^{(n)})|=\left\lfloor\frac{(n+1)^2}{4}\right\rfloor.
\]
\end{proof}

\section{Average weight}

\begin{theorem}\label{thm:mean}
The average weight over all $2^n$ binary sequences of length $n$ is
\[
\frac{1}{2}\binom{n+1}{2}.
\]
\end{theorem}

\begin{proof}
Choose $X$ uniformly from the $2^n$ binary sequences of length $n$. For a
fixed pair $0\leq u<v\leq n$, telescoping the prefix sums gives
\[
P_u+P_v=\sum_{r=u}^{v-1}a_r\pmod2.
\]
This sum contains at least one bit because $u<v$. If all coordinates except
$a_u$ are fixed, exactly one of the two choices $a_u=0$ and $a_u=1$ makes
the sum equal to $1$. Thus $P_u\neq P_v$ for exactly $2^{n-1}$ generators,
or with probability $1/2$.

Let $\mathbf{1}_{u,v}(X)$ be the indicator of the event $P_u\neq P_v$.
Proposition~\ref{prop:prefix} gives
\[
|D(X)|=\sum_{0\leq u<v\leq n}\mathbf{1}_{u,v}(X).
\]
Taking expectations and using linearity of expectation, without requiring
the indicators to be independent, we obtain
\[
\mathbb{E}|D(X)|
=\sum_{0\leq u<v\leq n}\mathbb{E}\mathbf{1}_{u,v}
=\binom{n+1}{2}\frac12.
\]
This expectation is exactly the average over all $2^n$ generators.
\end{proof}

\BLU{\begin{remark}
Since the triangle has $\binom{n+1}{2}$ entries, the average weight is exactly
half the number of entries: Delannoy--Steinhaus triangles are balanced
\emph{on average}, in sharp contrast with the rarity of exactly balanced
triangles established in Corollary~\ref{cor:balanced}.
\end{remark}}

\section{Conclusion}

In the present work, we introduced Delannoy--Steinhaus triangles over $\mathbb{Z}/2\mathbb{Z}$ and investigated their weight distribution. The oddness of the Delannoy coefficients allowed us to express every entry as the parity of a consecutive interval of the generating sequence. We then used prefix parities to reduce the weight of the entire triangle to a quantity determined only by the composition of the associated prefix-parity sequence.

This representation provides a unified approach to the main weight parameters. It yields the complete spectrum of attainable weights, an enumeration of the generating sequences realizing each value, and an explicit formula for triangles generated by canonical basis vectors. It also gives direct proofs of the minimum nonzero and {second-smallest nonzero} weights, the maximum weight, and the average weight over all binary generating sequences.

As a consequence of the complete spectrum, we also solved the balance problem
for binary Delannoy--Steinhaus triangles. Such a triangle generated by a
sequence of length $n$ exists if and only if $n+1$ is a perfect square. In
that case, the two admissible prefix compositions characterize all balanced
generators, and their total number is given explicitly by
equation~\eqref{eq:balanced-count}.

Several directions for future research naturally arise. The reduction used here depends essentially on the parity of the Delannoy coefficients, and extending the study to other moduli will require a finer analysis of their residue patterns. It would also be interesting to investigate more closely the structure of the generating sequences within each weight class, to study additional constrained families of generators, and {to determine whether natural graph-theoretic structures and invariants can be associated with these triangles}.

\section*{Funding}
The authors declare that no funds, grants, or other support were received during the preparation of this manuscript.

\section*{Competing interests}
The authors have no relevant financial or non-financial interests to disclose.

\section*{Declaration on the use of generative AI}
All mathematical content of this article is the authors’ own work. Generative AI (Claude, Anthropic) was used solely for language editing and formatting, under the authors’ full review and responsibility.

{\footnotesize\bibliographystyle{abbrv}
\bibliography{bibd}}
\end{document}